\documentclass[a4paper,12pt]{amsart}

\usepackage[utf8]{inputenc}
\usepackage[T1]{fontenc}
\usepackage{amsmath,amssymb,amsfonts,amsthm,mathrsfs}
\usepackage{geometry}
\usepackage{graphicx}
\usepackage{float}
\usepackage[english,french]{babel}
\usepackage{tikz}
\usetikzlibrary{arrows}
\usetikzlibrary{decorations.markings}

\theoremstyle{plain}
\newtheorem{thm}{Théorème}[section]
\newtheorem{df}[thm]{Définition}
\newtheorem{lem}[thm]{Lemme}
\newtheorem{prop}[thm]{Proposition}

\newtheorem{rmq}[thm]{Remarque}
\newtheorem{cor}[thm]{Corollaire}
\newtheorem*{thm*}{Théorème}
\newtheorem{pb}[thm]{Problème}
\newcommand{\mot}{}
\newtheorem*{thmref_interne}{\mot{}}
\newenvironment{thmref}[2]{
	\renewcommand{\mot}{#1 #2}
	\begin{thmref_interne}}
	{\end{thmref_interne}
}

\newcommand\mc{\mathcal}

\newcommand\mr{\mathrm}

\newcommand\F{\mc{F}}

\renewcommand\emph[1]{\textup{\textbf{#1}}}

\title{
Dualité complexe entre sous-variétés réelles dans $\mathbb{P}^2(\mathbb{C})$
}

\author{Olivier Thom}
\date{\today}
\thanks{L'auteur aimerait remercier l'IMPA et PNPD-CAPES pour le support qu'ils lui ont fourni}

\begin{document}

\begin{abstract}
On introduit une notion de sous-variété réelle semi-legendrienne dans une variété de contact complexe de dimension 3 et on prouve que les sous-variétés réelles de $\mathbb{C}^2$ ont un relevé unique dans $\mathbb{C}^3$.
On en déduit alors une dualité complexe entre les sous-variétés réelles de $\mathbb{P}^2(\mathbb{C})$.
\\
\\
\\
\textsc{Abstract.} We introduce a notion of semi-legendrian real submanifold in a complex contact manifold of dimension 3 and prove that real submanifolds of $\mathbb{C}^2$ can be uniquely lifted to $\mathbb{C}^3$.
Then we deduce a complex duality between real submanifolds of $\mathbb{P}^2(\mathbb{C})$.
\end{abstract}

\maketitle

\section{Introduction}

L'étude des germes de surfaces réelles dans $\mathbb{C}^2$ modulo biholomorphisme se révèle assez rapide : dans \cite[Chapitre I, I]{cartan_geometrie_pseudo_conforme_1}, E. Cartan montre qu'une surface analytique réelle est soit une courbe complexe, soit biholomorphe au voisinage d'un point générique à la surface $\mathbb{R}^2\subset \mathbb{C}^2$.
Cependant on ne saurait se satisfaire de ce résultat qui ne donne que peu d'informations sur le comportement d'une surface réelle vis-à-vis des sous-variétés holomorphes de l'espace ambient, même d'un point de vue local.

Dans cet état d'esprit, on se rappelle l'article \cite{thom_varietes_ordre_fini} de R. Thom, où l'auteur pose le problème suivant :

\begin{pb}
\label{pb_1}
Soit $V$ une variété réelle de dimension $2n$ compacte plongée dans $\mathbb{P}^{n+k}(\mathbb{C})$.
Supposons qu'il existe un ouvert dense $U$ dans la grassmannienne des k-plans complexes de $\mathbb{P}^{n+k}(\mathbb{C})$ tel que tout k-plan de $U$ intersecte $V$ en un nombre fixe de points.
Cette propriété implique-t-elle que $V$ soit une variété algébrique complexe ?
\end{pb}

Il enchaine en montrant qu'il suffit de considérer le cas où $V$ est une surface réelle $S$ dans $\mathbb{P}^2(\mathbb{C})$, puis donne une idée de preuve malheureusement incomplète.

Quelques années plus tard, W.F. Pohl publia un article prouvant ce résultat de manière plus satisfaisante (voir \cite{pohl_ordnungsgeometrie}).
En étudiant un peu plus le sujet on s'aperçoit qu'on peut en fait le voir comme une conséquence d'une dualité entre sous-variétés réelles de $\mathbb{P}^2(\mathbb{C})$ : à chaque surface réelle $S$ on peut associer l'ensemble dual $\check{S}\subset \check{\mathbb{P}}^2(\mathbb{C})$ des droites complexes qui intersectent $S$ de manière non transverse.
Cet ensemble est en général une hypersurface réelle, et le problème \ref{pb_1} peut être vu comme portant sur cette hypersurface.

Cette construction n'est pas sans rappeler la dualité entre courbes dans $\mathbb{P}^2$.
Cette dernière se comporte très bien au sens où le bidual d'une courbe $C$ est en général $C$ elle-même, puisque l'enveloppe de la famille des droites tangentes à $C$ est égale à $C$.
Dans le cas présent, il n'y a pas de construction inverse évidente, et il n'est pas trivial a priori que deux surfaces ne puissent pas donner le même dual (c'est le problème de droites dites bitangentes dans \cite{thom_varietes_ordre_fini} et \cite{pohl_ordnungsgeometrie}, que l'on devrait plutôt qualifier de bicritiques pour ne pas confondre avec les droites tangentes).

Dans cet article, on va prouver que cette dualité entre sous-variétés réelles est en fait plus générale en introduisant une notion de sous-variété semi-legendrienne dans une variété de contact complexe de dimension 3, et une construction qui permet de relever une sous-variété réelle de $\mathbb{C}^2$ en une sous-variété semi-legendrienne de $\mathbb{C}^3$.

On commencera en partie \ref{sec_rappels} par rappeler quelques faits sur les plans réels linéaires et les intersections locales entre surfaces réelles et droites complexes.

Dans la partie \ref{sec_exceptions}, on présentera quelques surfaces qui se comportent de manière exceptionnelles.
Celles-ci constitueront les cas les plus dégénérés pour la dualité, d'où l'utilité de bien les comprendre.

Puis, dans la section \ref{sec_semi_legendrienne}, on définira puis étudiera les sous-variétés semi-legendriennes.
Si $p$ désigne la fibration $p: \mathbb{P}(T \mathbb{P}^2) \rightarrow \mathbb{P}^2$, on construira des relevés $p^*N$ semi-legendriens pour des sous-variétés lisses $N$ de dimension $\leq 3$.
Pour obtenir un résultat global sur les sous-variétés avec des points singuliers, on introduira une stratification sur $N$ vérifiant une condition (C) (une version forte de la condition (A) de Whitney) et on prouvera les résultats suivants.

\begin{thmref}{Théorème}{\ref{thm_stratification_1}}
Soit $N$ une sous-variété réelle de dimension inférieure ou égale à 3 dans $\mathbb{P}^2(\mathbb{C})$, de régularité $\mc{C}^2$.
Supposons que $N$ admette une stratification $N=\cup N_i$ vérifiant la condition (C).
Alors $p^*N$ est une sous-variété de régularité $\mc{C}^1$ dans $\mathbb{P}(T \mathbb{P}^2(\mathbb{C}))$ et admet une stratification dont les strates de dimension 3 sont semi-legendriennes.
De plus, $p(p^*N)=N$.
\end{thmref}

\begin{thmref}{Théorème}{\ref{thm_stratification_2}}
Soit $M$ une sous-variété réelle de dimension trois, de régularité $\mc{C}^2$ dans $\mathbb{P}(T \mathbb{P}^2(\mathbb{C}))$.
On suppose que $N:=p(M)$ est une sous-variété de $\mathbb{P}^2(\mathbb{C})$, et que $M$ et $N$ admettent des stratifications $M=\cup M_i$ et $N=\cup N_i$ de sorte que les strates de dimension 3 de $M$ soient semi-legendriennes et que $(N_i)$ vérifie la condition (C).
On suppose de plus que $p$ envoie toute strate de $M$ dans une strate de $N$, et est une submersion en restriction à chaque strate.
Alors $M\subset p^*(p(M))$.
\end{thmref}

Finalement, dans la section \ref{sec_dualite} on introduira la dualité entre sous-variétés réelles de $\mathbb{P}^2(\mathbb{C})$ et sous-variétés réelles de $\check{\mathbb{P}}^2(\mathbb{C})$.
Le lemme suivant sera une conséquence rapide des théorèmes précédents.

\begin{thmref}{Lemme}{\ref{lem_bidualite_generique}}
Soit $N$ un germe de sous-variété lisse générique de régularité $\mc{C}^2$ dans $\mathbb{P}^2(\mathbb{C})$.
Alors son dual $\check{N}\subset \check{\mathbb{P}}^2(\mathbb{C})$ est lisse de régularité $\mc{C}^1$ et vérifie $N=\check{\check{N}}$.
\end{thmref}

En application de toutes ces notion, on proposera en fin de section \ref{sec_dualite} une preuve du problème \ref{pb_1}.
Un énoncé plus précis de ce résultat, en dimension et codimension quelconques est le suivant.

\begin{thmref}{Théorème}{\ref{cor_principal}}
Soit $V$ une sous-variété réelle connexe compacte irréductible de dimension réelle $2n$ et de régularité $\mc{C}^2$ de $\mathbb{P}^{n+k}(\mathbb{C})$, et $G$ la grassmannienne des k-plans linéaire complexes.
Alors ou bien
\begin{enumerate}
\item $V$ est une sous-variété algébrique complexe, ou
\item $V$ est le compactifié d'une sous-variété affine réelle de $\mathbb{C}^{n+k}$, ou
\item il existe deux ouverts non vides $U_1, U_2\subset G$ tels que tout plan de $U_i$ intersecte $V$ transversalement en $n_i$ points, et $n_1\neq n_2$.
\end{enumerate}
\end{thmref}

\section{Rappels sur les surfaces réelles}
\label{sec_rappels}

\subsection{Plans réels}

Considérons un 2-plan réel linéaire $P\subset \mathbb{C}^2$, et des coordonnées holomorphes $(x,y)=(x_1+ix_2,y_1+iy_2)$ dans $\mathbb{C}^2$.
On considèrera que la base $(\frac{\partial}{\partial x_1},\frac{\partial}{\partial x_2},\frac{\partial}{\partial y_1},\frac{\partial}{\partial y_2})$ est directe et on notera $J: \mathbb{C}^2 \rightarrow \mathbb{C}^2$ l'opérateur de multiplication par $i$.

Prenons une base orthonormale $(t_1,t_2)$ de $P$, et une base orthonormale $(n_1,n_2)$ de $P^\perp$ telles que $(t_1,t_2,n_1,n_2)$ soit une base directe.
Le vecteur $Jt_1$ est orthogonal à $t_1$ donc se décompose
\[
Jt_1 = \mr{cos}(\theta)t_2 + \mr{sin}(\theta)n,
\]
où $n\in P^\perp$ et $\theta\in [0,\pi]$.
Cet angle $\theta$ ne dépend en fait que du plan orienté $P$ et s'appelle l'\emph{angle holomorphe} de $P$ (ou angle de Wirtinger).
Le vecteur $n$ correspond à la projection sur $P^\perp$ de $Jt_1$, et on voit que l'application $t_1 \mapsto n$ est une isométrie indirecte de $P$ vers $P^\perp$.

Le plan $P$ permet de déterminer une unique structure complexe $j$ telle que $P$ soit une direction complexe et $t_2=jt_1$.
En effet $j$ est entièrement déterminée sur $P$ ; en outre $P^\perp$ doit aussi être une direction complexe donc stable par $j$.
Comme $(t_1,t_2,n_1,n_2)$ est une base directe, on doit avoir $jn_1=n_2$, ce qui permet de conclure.

En fait, la grassmannienne $\mr{Gr} := \mr{G}_{\mathbb{R}}^+(2,4)$ des 2-plans réels orientés est homéomorphe à un produit de deux sphères.
La première sphère représente la structure complexe $j$ et la deuxième correspond à la droite projective complexe associée.
L'angle complexe $\theta$ correspond alors à l'angle dans la première sphère entre $j$ et $J$ (voir \cite[\S 4]{chen+morvan_totally_real_surfaces}).

Revenons au 2-plan réel $P$.
Une droite complexe $D$ peut se comporter de trois façons par rapport à $P$ : il se peut que $D$ soit tangente à $P$, i.e. $D=P$ ; il se peut que $D$ soit transverse à $P$, i.e. $P\oplus D=\mathbb{C}^2$ ; et il se peut que $D\cap P = \mathbb{R} v$ pour un vecteur $v$.
Caractérisons le cas où $D$ et $P$ ne sont pas transverses.
Supposons que $D$ ait pour équation $y=\lambda x$ et que $P$ soit défini par 
\[
\left\{ \begin{aligned}
f_1 &= \alpha_1 x_1 + \beta_1 x_2 + \gamma_1 y_1 + \delta_1 y_2\\
f_2 &= \alpha_2 x_1 + \beta_2 x_2 + \gamma_2 y_1 + \delta_2 y_2.
\end{aligned} \right.
\]
Introduisons 
\[
\Omega_P := df_1 \wedge df_2
\]
et
\[
\Omega_{\lambda} := \frac{-1}{2i}(dy-\lambda dx) \wedge (d\bar{y} - \bar{\lambda}d\bar{x}).
\]
Les plans $D$ et $P$ ne sont pas transverses si et seulement si 
\begin{equation}
\label{eq_cercle_critique}
\begin{aligned}
0 &= \Omega_P \wedge \Omega_{\lambda}\\
&= (\alpha_1 \beta_2 - \alpha_2 \beta_1) + \lambda_1 (\alpha_1 \delta_2 - \alpha_2 \delta_1 + \beta_1 \gamma_2 - \beta_2 \gamma_1)\\
&\quad - \lambda_2 (\beta_1 \delta_2 - \beta_2 \delta_1 + \alpha_1 \gamma_2 - \alpha_2 \gamma_1) + |\lambda|^2 (\gamma_1 \delta_2 - \gamma_2 \delta_1).
\end{aligned}
\end{equation}

On voit ainsi que l'ensemble des directions complexes non transverses est un cercle réel sur la sphère $\mathbb{P}^1(\mathbb{C})$ (rappelons que la notion de cercle est stable par transformation de Möbius).
On nommera \emph{cercle critique} de $P$ l'ensemble des directions $\lambda\in \mathbb{P}^1(\mathbb{C})$ telles que la droite $D$ n'est pas transverse à $P$.

De plus, le rayon d'un cercle est bien défini modulo les transformations de Möbius obtenues par projectivisation de $\mr{U}_2(\mathbb{C})$ ; modulo l'action de ce groupe, on peut supposer que $f_1=y_2$ et $f_2 = \mr{cos}(\theta)y_1+\mr{sin}(\theta)x_2$.
On obtient alors l'équation 
\begin{equation}
\label{eq_non_transverse}
\mr{cos}(\theta) |\lambda|^2 = \mr{sin}(\theta) \lambda_2.
\end{equation}

On voit donc que le rayon du cercle critique est determiné par $\theta$.

\begin{figure}[H]
\label{fig_cercle_critique}
\begin{center}
\begin{tikzpicture}
\draw[>=stealth,->] (-2,0) -- (2,0);
\draw (2,0) node[above] {$\lambda_1$};
\draw[>=stealth,->] (0,-1) -- (0,3);
\draw (0,3) node[right] {$\lambda_2$};
\draw (0,1) circle (1);

\draw (0,0) node {$\bullet$};
\draw (0,-0.5) node[below,right] {$\lambda=0$};
\draw (0,2) node {$\bullet$};
\draw (0.5,2) node[above,right] {$\lambda=i\mr{tan}(\theta)$};
\end{tikzpicture}
\end{center}
\caption{Le cercle critique après changement de coordonnées.}
\end{figure}
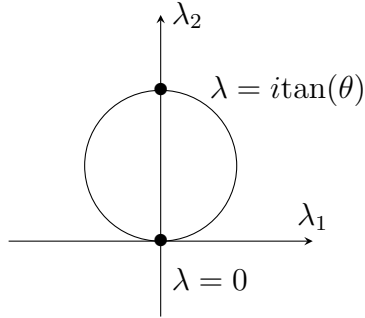

\subsection{Intersection d'une surface et d'une droite complexe}

Soit $(S,0)$ un germe de surface réelle générique de $\mathbb{C}^2$ de régularité $\mc{C}^2$ et $D_0$ une droite complexe de pente $\lambda$ passant par l'origine.
On suppose que $D_0$ intersecte $S$ non transversalement à l'origine.
On s'intéresse à l'intersection entre un translaté $D_\mu$ de $D_0$ et $S$.
Dans l'espace des paramètres $\mu\in \mathbb{C}$, on voit que pour $S$ générique, l'ensemble des paramètres $\mu$ tels que $D_\mu$ intersecte $S$ non transversalement est une courbe réelle et qu'elle sépare le germe d'espace $(\mathbb{C}_\mu,0)$ en deux régions où l'intersection vaut respectivement 0 et 2.
Ceci peut être vu par un calcul direct, par des considérations liés à la projection $p_{\lambda}:\mathbb{C}^2 \rightarrow \mathbb{C}$ comme dans \cite{thom_varietes_ordre_fini}, ou via les droites sécantes, comme le propose \cite{pohl_ordnungsgeometrie}.

Il en résulte que pour une direction $\lambda\in \mathbb{P}^1(\mathbb{C})$, l'ensemble $C_{\lambda}$ des points de $S$ pour lesquels $\lambda$ est une direction non transverse est génériquement une courbe.

Dans son article, W. Pohl a mis en évidence que certains couples $(p,\lambda)$ où $\lambda$ est une direction critique en $p\in S$ se comportaient de manière exceptionnels (les points de type F selon sa terminologie).
Donnons quelques caractérisations équivalentes de ces points.

\begin{lem}
\label{lem_exceptionnel}
Soit $(S,0)$ un germe de surface réel lisse de $\mathbb{C}^2$ tel que $T_0S$ ne soit pas complexe.
Considérons une droite $d$ passant par l'origine dans une direction $\lambda$ critique.
Notons $v$ un vecteur tangent à $S$ à l'origine tel que $\mr{Vect}_{\mathbb{C}}(v) = d$.
Les conditions suivantes sont équivalentes :
\begin{enumerate}
\item\label{lcsse_item_1} Le vecteur de courbure $k$ d'une courbe $(C,0)\subset (S,0)$ passant à l'origine dans la direction $v$ vérifie $k\in d$.
\item\label{lcsse_item_2} La courbe $C_{\lambda}$ est singulière ou vérifie $T_0C_{\lambda} = v$.
\item\label{lcsse_item_3} La courbe $C_{\lambda}$ est singulière ou admet une paramétrisation régulière $c(t)$ telle que la droite complexe $d(t)$ passant en $c(t)$ dans la direction $\lambda$ ait pour équation $y=\lambda x + \mu(t)$ avec $\mu'(0)=0$.
\item\label{lcsse_item_4} Pour une courbe $c(t)$ sur $S$ passant à l'origine dans la direction $v$, le point $\lambda$ sur $\mathbb{P}^1(\mathbb{C})$ est un point fixe à l'ordre 1 de la famille des cercles critiques en $c(t)$.
\end{enumerate}
\end{lem}

\begin{proof}
On peut supposer que $d=\{y=0\}$ et $v=\frac{\partial}{\partial x_1}$.
Supposons pour prouver \eqref{lcsse_item_2} $\Leftrightarrow$ \eqref{lcsse_item_3} que la courbe $C_{\lambda}$ soit régulière.
Considérons une paramétrisation régulière $c(t)=(at,bt) + O(t^2)$ de $C_{\lambda}$.
L'ordonnée à l'origine $\mu(t)$ d'une droite horizontale passant par $c(t)$ vérifie alors $\mu(t)=bt + O(t^2)$ prouvant \eqref{lcsse_item_3} $\Leftrightarrow b=0$.
Mais si $b=0$, alors $(a,0)$ est un vecteur tangent à $S$ ; comme $d\cap T_0S = \langle v \rangle$, on a bien \eqref{lcsse_item_2} $\Leftrightarrow$ \eqref{lcsse_item_3}.

Considérons des champs de 2-plans $\Omega_S$ et $\Omega_{\lambda}$ définissant respectivement un champ de plans tangent à $S$ et le feuilletages des droites complexes de direction $\lambda$.
Alors la fonction $E$ telle que $\Omega_S \wedge \Omega_{\lambda} = E dV$ est une équation de $C_{\lambda}$, et la condition \eqref{lcsse_item_2} est équivalente à $dE(v)=0$.

Considérons une courbe $c(t)=vt+kt^2+O(t^3)$ sur $S$ dans la direction $v$.
Le vecteur $c'(t)=v+2kt+O(t^2)$ est tangent à $S$ en $c(t)$ donc la direction $\lambda(t)$ complexifiée de $c'(t)$ vérifie $\lambda(t) = \lambda + O(t^2)$ si et seulement si $k\in d$.
D'un autre côté, $\lambda(t)=\lambda + O(t^2)$ si et seulement si $dE(v)=0$ donc on obtient \eqref{lcsse_item_2} $\Leftrightarrow$ \eqref{lcsse_item_1}.

Les conditions $dE(v)=0$ et \eqref{lcsse_item_4} sont clairement équivalentes.
\end{proof}

\begin{df}
\label{df_exceptionnel}
Un couple $(p,\lambda)$ tel que $\lambda\in \mathbb{P}^1(\mathbb{C})$ soit une direction critique en $p\in S$ sera appelée \emph{exceptionnel} s'il vérifie l'une des conditions du lemme \ref{lem_exceptionnel}.
\end{df}

\section{Surfaces exceptionnelles}
\label{sec_exceptions}


\subsection{Surfaces totalement exceptionnelles}

On s'intéresse ici aux germes de surfaces $S$ de plan tangent à l'origine non complexe pour lesquelles tout couple $(p,\lambda)$ avec $p\in S$ et $\lambda$ critique est exceptionnel.
On a vu au lemme \ref{lem_exceptionnel} que $(p,\lambda)$ exceptionnel signifie que la courbe $C_{\lambda}$ est tangente au feuilletage des droites de direction $\lambda$.
Ceci implique que $C_{\lambda}$ est incluse dans une unique droite complexe de direction $\lambda$.
On obtient ainsi la proposition suivante.

\begin{prop}
\label{prop_U_0}
Supposons que tout couple $(p,\lambda)$ où $\lambda$ est dans le cercle critique en $p\in S$ soit exceptionnel.
Alors toute droite complexe non transverse à $S$ coupe $S$ le long d'une courbe.
\end{prop}

\subsection{Feuilletages affines complexes non transverses}

Au vu de la proposition \ref{prop_U_0}, les surfaces ayant beaucoup de points exceptionnels vont avoir des courbes incluses dans l'intersection entre $S$ et des droites complexes.
On cherche dans cette section à caractériser les surfaces telles que ces courbes exceptionnelles forment des feuilletages.

\begin{thm}
Supposons qu'il existe un feuilletage lisse $\F$ sur $S$ dont toute feuille est incluse dans une droite complexe.
Alors $S$ est incluse dans une hypersurface Levi-plate feuilletée par des droites complexes.
\end{thm}

\begin{proof}
Les feuilles de $\F$ forment une famille à un paramètre réel de courbes réelles $(F_t)$ ; chaque $F_t$ est incluse dans une droite complexe $D_t$.
Les $F_t$ ne s'intersectant pas sur $S$, la courbe des points d'intersection $D_t\cap D_{t'}$ est loin de $S$.
On voit donc qu'il existe un voisinage $U$ de $S$ dans $\mathbb{C}^2$ tel que les $D_t\cap U$ soient deux à deux disjoints, et donc l'hypersurface $H=U\cap\bigcup D_t$ convient.
\end{proof}

\begin{thm}
\label{thm_2}
Soit $S$ un germe de surface réelle de $\mathbb{P}^2(\mathbb{C})$.
Supposons qu'il existe deux pinceaux de droites $\mc{P}_1,\mc{P}_2$ et deux feuilletages lisses transverses $\F_1,\F_2$ sur $S$ tels que toute feuille de $\F_i$ soit incluse dans une droite du pinceau $\mc{P}_i$.
Alors il existe deux courbes réelles $C_1,C_2\subset \mathbb{C}$ et un automorphisme holomorphe $\varphi$ de $\mathbb{P}^2(\mathbb{C})$ tels que $S\subset\varphi(C_1\times C_2)$ où $C_1\times C_2$ est vu comme une surface réelle de $\mathbb{C}^2$ et $\mathbb{C}^2$ est considéré comme inclus dans $\mathbb{P}^2(\mathbb{C})$.
\end{thm}

\begin{proof}
Quitte à appliquer un automorphisme de $\mathbb{P}^2$, on peut supposer que les centres des pinceaux sont aux points de coordonnées $[0:1:0]$ et $[1:0:0]$.
Les pinceaux sont alors $\mc{P}_1=\{x=t\}$ et $\mc{P}_2=\{y=t\}$ ; notons $T_1=\mathbb{C}_x$ et $T_2=\mathbb{C}_y$ les espaces des paramètres de ces pinceaux.
Les droites de ces pinceaux qui intersectent $S$ forment des courbes réelles $C_i\subset T_i$.
Chaque point de $S$ appartient à une droite de chaque pinceau donc définit un point $(t_1,t_2)\in C_1\times C_2\subset T_1\times T_2$.
Le résultat s'ensuit.
\end{proof}

\begin{thm}
\label{thm_3}
Soit $S$ un germe de surface réelle de $\mathbb{P}^2(\mathbb{C})$.
Supposons qu'il existe trois pinceaux de droites $\mc{P}_i$ et trois feuilletages lisses deux à deux transverses $\F_i$ sur $S$ tels que toute feuille de $\F_i$ soit incluse dans une droite du pinceau $\mc{P}_i$.
Alors $S$ est incluse dans un plan affine réel.
\end{thm}

\begin{proof}
Quitte à appliquer un automorphisme de $\mathbb{P}^2$, on peut supposer que les centres des pinceaux sont à l'infini.
Alors à chaque pinceau correspond une pente $\lambda_i\in \mathbb{P}^1(\mathbb{C})$, et on voit que chaque $\lambda_i$ appartient au cercle critique en tout point $p\in S$.
On en déduit immédiatement que le cercle critique est constant le long de $S$.

On peut supposer $\lambda_1=0, \lambda_2=\infty$ et on a vu au théorème \ref{thm_2} que cela implique $S\subset C_1\times C_2\subset \mathbb{C}^2$.
Pour simplifier, supposons que la courbe $C_1$ ait pour équation $x_2=a_1x_1^2+\ldots$ et que l'équation de $C_2$ soit $y_2=a_2y_1^2 + \ldots$
La surface $S$ a alors pour équations
\[
\left\{ \begin{aligned}
x_2&=a_1x_1^2+\ldots\\
y_2&=a_2y_1^2+\ldots
\end{aligned} \right.
\]
À l'origine le cercle critique est $\mathbb{P}^1(\mathbb{R})$. 
Si l'un des coefficients $a_i$ est non nul (disons par exemple $a_2$), on voit qu'au point $(x,y)=(0,\varepsilon)$, le plan tangent est engendré par $\frac{\partial}{\partial x_1}$ et $\frac{\partial}{\partial y_1}+a_2 \varepsilon_1 \frac{\partial}{\partial y_2}$.
Le vecteur tangent $\frac{\partial}{\partial x_1}+\frac{\partial}{\partial y_1}+a_2 \varepsilon_1 \frac{\partial}{\partial y_2}$ a pour direction complexe $\lambda=1+a_2 \varepsilon_1 i$ qui n'est pas réel, contredisant la constance du cercle critique.

Ainsi les deux courbes $C_1,C_2$ sont de courbure nulle en tout point, donc sont des droites réelles, ce qui conclut la preuve.
\end{proof}

On en déduit le corollaire suivant (Proposition A dans \cite{pohl_ordnungsgeometrie}).

\begin{cor}
\label{cor_R2}
Soit $(S,p_0)$ un germe de surface réelle de $\mathbb{P}^2(\mathbb{C})$ tel que $T_{p_0}S$ ne soit pas complexe et tel que tout couple $(p,\lambda)$ critique soit exceptionnel.
Alors $S$ est incluse dans un plan affine réel.
\end{cor}

\begin{proof}
On peut supposer que le plan tangent n'est complexe pour aucun point de $S$.
Par chaque point $p\in S$ passe alors un pinceau de courbes $d\subset S$ telles que $d$ soit incluse dans une droite affine complexe.
Prenons trois points génériques $p_1,p_2,p_3\in S$ ; ils définissent trois pinceaux de droites $\mc{P}_i$ dans $\mathbb{P}^2(\mathbb{C})$ et trois feuilletages $\mc{F}_i$ sur $S$ que l'on peut utiliser pour appliquer le théorème \ref{thm_3}.
\end{proof}

\section{Variétés semi-legendriennes}
\label{sec_semi_legendrienne}

Considérons l'espace $\mathbb{C}^3$ muni de coordonnées holomorphes $(x,y,\lambda)$ et d'une forme de contact $\omega := dy - \lambda dx$ ; on notera $\mc{P}$ le champ de 2-plans complexes associé et $p(x,y,\lambda)=(x,y)$ la projection vers $\mathbb{C}^2$.

On cherche à définir des sous-variétés réelles $M\subset \mathbb{C}^3$ équivalentes aux courbes legendriennes, à étudier les projections $p(M)$ pour ces sous-variétés, et à définir des relevés $p^*N$ pour toutes les sous-variétés $N$ de $\mathbb{C}^2$.
De manière assez surprenante, on peut donner des définitions indépendantes de la dimension. 

\begin{df}
Une sous-variété réelle $M$ de dimension 2 ou 3 et de régularité $\mc{C}^1$ dans $(\mathbb{C}^3,\omega)$ sera dite \emph{semi-legendrienne} si en tout point lisse $w\in M$, on a $\mr{dim}(\mc{P}\cap T_wM)\geq 2$.
\end{df}

\begin{df}
Si $N$ est une sous-variété réelle lisse de $\mathbb{C}^2$ de dimension $\leq 3$ et de régularité $\mc{C}^1$, on définit le relevé $p^*N$ de $N$ par la propriété que $p^*N\cap p^{-1}(z)$ soit l'ensemble des directions $\lambda$ des droites complexes $D$ passant par $z\in N$ telles que $D+T_zN\neq \mathbb{C}^2$.
\end{df}

\subsection{Hypersurfaces}

Soit $H$ une hypersurface réelle de $\mathbb{C}^2$ de régularité $\mc{C}^1$ telle que la pente $\lambda(z)$ de l'unique droite complexe $T^{1,0}_zH$ incluse dans $T_zH$ ne soit pas verticale.
On voit que la variété $p^*H\subset \mathbb{C}^3$ est de régularité $\mc{C}^0$ et définie par
\[
p^*H = \{(z,\lambda(z)), z\in H\}.
\]

\begin{lem}
\label{lem_hyp+legendrien_1}
Soit $M$ un germe de sous-variété semi-legendrienne en $(x,y,\lambda)$, lisse, de régularité $\mc{C}^1$ et transverse à la projection $p: \mathbb{C}^3 \rightarrow \mathbb{C}^2$.
Alors $p(M)$ est un germe d'hypersurface lisse de régularité $\mc{C}^1$ de $\mathbb{C}^2$ et $T^{1,0}_{(x,y)}p(M)$ a pour direction complexe $\lambda$.
Ainsi $p^*(p(M))=M$.
\end{lem}

\begin{proof}
Notons $z=(x,y)$ et $w=(z,\lambda)$.
La fibre $F=p^{-1}(z)$ est incluse dans le plan $\mc{P}_{w}$ donc si $M$ est transverse à la projection, on a $F\oplus (TM\cap \mc{P}) = \mc{P}$ au point $w$.
Ainsi, si $C$ est une courbe holomorphe legendrienne passant par $w$, transverse à la projection, les deux plans $T_{w}C$ et $T_{w}M\cap \mc{P}$ ont la même image par $p$.
Comme $p(T_{w}C)$ est complexe de pente $\lambda$, le résultat est prouvé.
\end{proof}

\begin{lem}
Si $(H,0)$ est un germe d'hypersurface réelle lisse de régularité $\mc{C}^2$ dans $\mathbb{C}^2$, alors $p^*H$ est un germe de sous-variété semi-legendrienne lisse de régularité $\mc{C}^1$ dans $(\mathbb{C}^3,\omega)$, transverse à la projection $p$.
De plus, $p(p^*H) = H$.
\end{lem}

\begin{proof}
Que $p^*H$ soit lisse et transverse à $p$ est évident puisque $\lambda(z)$ est une fonction $\mc{C}^1$ ; il suffit donc de montrer que $p^*H$ est semi-legendrienne.
Sans perte de généralité, on peut supposer que $\lambda(0)=0$ ; on peut donc trouver deux courbes réelles $c_1(t),c_2(t)$ sur $H$ passant par l'origine avec une dérivée égale respectivement à $\frac{\partial}{\partial x_1}$ et $\frac{\partial}{\partial x_2}$.
La courbe $c_1$ admet le développement $c_1(t) = (t,0) + O(t^2)$ donc son relevé $c_1^*$ sur $p^*H$ admet un développement $c_1^*(t)=(t,0,l_1t)+O(t^2)$.
Le vecteur $T_0c_1^* = (1,0,l_1)$ est tangent à $p^*H$ et contenu dans $\mc{P}$.
De la même manière, le vecteur $T_0c_2^* = (i,0,l_2)$ donne un deuxième vecteur tangent à $p^*H$, transverse au premier et contenu dans $\mc{P}$, prouvant le caractère semi-legendrien de $p^*H$.
L'égalité $p(p^*H)=H$ est immédiate.
\end{proof}

\subsection{Surfaces}

Considérons une surface réelle $S\subset \mathbb{C}^2$ de régularité $\mc{C}^1$.
Dans ce cas la variété réelle $p^*S\subset \mathbb{C}^3$ de régularité $\mc{C}^0$ est définie par le fait que $p^*S\cap p^{-1}(z)$ soit le cercle critique au dessus de $z\in S$.

\begin{lem}
\label{lem_surf+legendrien_1}
Soit $M$ un germe de sous-variété semi-legendrienne lisse en $(x,y,\lambda)$.
Supposons qu'il existe un feuilletage lisse $\F$ par courbes sur $M$ tel que chaque feuille de $\F$ soit contenue dans une fibre de $p$ et que la projection $p$ soit transverse à l'espace des feuilles. 
Alors $p(M)$ est un germe de surface lisse de $\mathbb{C}^2$ et $\lambda$ est une direction critique en $(x,y)$.
Ainsi $M\subset p^*(p(M))$.
\end{lem}

\begin{proof}
D'après les hypothèses, la dimension chute par la projection et $p(M)$ est un germe de surface lisse de $\mathbb{C}^2$.
Écrivons $z=(x,y)$ et $w=(x,y,\lambda)$ pour simplifier.
Considérons un vecteur $v\in T_wM$ non tangent à $\F$ et une courbe holomorphe legendrienne $C$ passant par $w$ dans la direction $v$.
Alors $C$ est transverse à $p$ ; le vecteur $p_*(v)$ appartient à $T_zp(C)\cap T_zp(M)$ et la pente complexe de $p(C)$ en $z$ est $\lambda$.
Le résultat s'ensuit.
\end{proof}

\begin{lem}
Soit $(S,0)$ un germe de surface réelle lisse de régularité $\mc{C}^2$ dans $\mathbb{C}^2$.
Supposons que $T_0S$ ne soit pas une direction complexe, alors $p^*S$ est un germe de sous-variété semi-legendrienne lisse de $(\mathbb{C}^3,\omega)$, et la famille $\F$ des cercles critique donne un feuilletage lisse dont l'espace des feuilles est transverse à la projection $p$.
De plus, $p(p^*S)=S$.
\end{lem}

\begin{proof}
Si $T_0S$ n'est pas complexe, alors le plan tangent $TS$ n'est pas complexe dans tout un voisinage de l'origine.
On en déduit tout de suite que $p^*S$ est lisse de dimension 3, que le feuilletage $\F$ des cercles critique est lisse et que l'espace des feuilles est transverse à $p$.

Supposons que $\lambda=0$ soit une direction critique à l'origine.
Soit $v\in T_0S$ la direction tangente dont le complexifié est $\lambda=0$ et $c(t)$ une courbe tangente à $S$ passant par l'origine dans la direction $v$.
Considérons une famille $\mc{C}^1$ de directions critiques $\lambda(t)$ en $c(t)$ telle que $\lambda(0)=0$, et $c^*(t) := (c(t),\lambda(t))\in p^*S$ le relevé correspondant.
Alors $c^*(t) = (vt,0,\lambda'(0)t) + O(t^2)$ et le 2-plan complexe $\mc{P}_0$ contient les deux vecteurs $T_0\F$ et $(v,0,\lambda'(0))$ tangents à $p^*S$.
On en déduit que $p^*S$ est bien semi-legendrien.
La propriété $p(p^*S)=S$ est une conséquence immédiate des définitions.
\end{proof}

\begin{prop}
Soit $M$ une sous-variété semi-legendrienne de dimension 2 dans $\mathbb{C}^3$, alors $M$ est une courbe holomorphe legendrienne.
\end{prop}

\begin{proof}
On remarque qu'une courbe $c(t)$ dans $\mathbb{C}^2$ a un unique relevé $w(t)=(c(t),\lambda(t))$ où $\lambda(t)$ est le complexifié de $c'(t)$.
Si $M$ est contenue dans une fibre de $p$ elle est holomorphe ; si $p$ définit une fibration de rang générique 1 sur $M$, alors $p(M)$ est une courbe $c(t)$, mais alors au voisinage d'une point $w=(c(t),\lambda)$ avec $\lambda\neq \lambda(t)$, on a $T_wM\not\subset \mc{P}$.
Finalement, si $M$ est transverse à $p$ en $w$, prenons deux vecteurs tangents $v_1,v_2$ et deux courbes $w_i(t)$ tangentes à $v_i$ en $w$.
Chaque $w_i$ est égale au relevé de $p(w_i)$, donc $\lambda_i(t)$ est le complexifié de $p(w_i)'(t)$.
On en déduit que $v_1$ et $v_2$ ont le même complexifié donc que $T_wM$ est une droite complexe.
\end{proof}

\subsection{Courbes et points}

Si la variété réelle $N$ est une courbe ou un point dans $\mathbb{C}^2$, la variété $p^*N$ est le fibré trivial : 
\[
p^*N := p^{-1}(N).
\]

Réciproquement, on voit immédiatement que si les fibres de la projection $p$ sont tangentes à une sous-variété semi-legendrienne $M$, alors $p(M)$ est une courbe réelle ou un point dans $\mathbb{C}^2$.
Les égalités $p(p^*N)=N$ et $p^*(p(M))=M$ sont triviales dans ce contexte.

\subsection{Situation globale}
\label{sec_stratifications}

Dans le contexte global, on remplace $\mathbb{C}^2$ par $\mathbb{P}^2(\mathbb{C})$ et $\mathbb{C}^3$ par le projectivisé du fibré tangent $\mathbb{P}(T \mathbb{P}^2(\mathbb{C}))$.
Ceci permet en particulier d'éviter les problèmes de directions verticales.

Si on cherche à obtenir un résultat global pour unifier les résultats des sous-sections précédentes, il devient nécessaire de gérer les points singuliers et les points critiques de la projection $p$.

Les définitions sont faites de telle sorte que si l'ensemble singulier $S$ d'une sous-variété $N\subset \mathbb{P}^2(\mathbb{C})$ est de dimension au plus 1, alors le relevé $p^*S$ contient toute limite de points génériques de $p^*N$, donc ces points ne posent aucun problème.
Cependant, si $S$ est de dimension 2, il se peut que les limites de points de $p^*N$ constituent un ensemble de dimension 4 dans $\mathbb{P}(T \mathbb{P}^2)$.

Il semble donc naturel de se restreindre au contexte où $N$ admet une stratification de Whitney.
La notion correspondante dans $\mathbb{P}(T \mathbb{P}^2)$ est celle de sous-variétés $M$ admettant des stratifications dont les strates de grande dimension sont semi-legendriennes.
Cependant, sans hypothèses de régularité supérieure, la projection par $p$ d'un espace admettant une stratification de Whitney n'admet pas en général de stratification de Whitney ; elle n'a même aucune raison d'être localement finie.

Il semble donc que le contexte dans lequel la dualité globale soit naturelle reste encore à trouver, mais énonçons déjà les résultats globaux que l'on peut obtenir en l'état, ne serait-ce que pour unifier les sections précédentes.

Commençons par rappeler les définitions pour éviter les confusions.
Si $X$ est une variété $\mc{C}^\infty$ et $Y\subset X$ une sous-variété singulière de régularité $\mc{C}^\mu$, on appelle \emph{stratification} de $Y$ une décomposition disjointe $Y=\cup Y_i$ où chaque $Y_i$ est une sous-variété $\mc{C}^\mu$ lisse (les strates) de sorte que que chaque point $y\in Y$ admette un voisinage qui ne rencontre qu'un nombre fini de strates et que si $Y_i \cap \overline{Y_j}\neq \emptyset$, alors $\mr{dim}(Y_i) < \mr{dim}(Y_j)$ (dans ce cas on écrira $Y_i < Y_j$).
Introduisons la condition suivante sur les stratifications :
\begin{equation}
\begin{aligned}
&\text{Pour toutes strates $Y_i<Y_j$, et toute suite de points $y_n\in Y_j$ telle que}\\ &\quad\text{ $y_n \to y_\infty\in Y_i$, alors $T_{y_n}Y_j$ converge et sa limite contient $T_{y_\infty}Y_i$.}
\end{aligned}\tag{C}
\end{equation}

On remarque que la condition $(C)$ implique évidemment la condition $(A)$ de Whitney.
Si $N$ admet une stratification $N=\cup N_i$, on notera $p^*N = \cup p^*N_i$.
Cette définition dépend du choix de la stratification.

\begin{thm}
\label{thm_stratification_1}
Soit $N$ une sous-variété réelle de dimension inférieure ou égale à 3 dans $\mathbb{P}^2(\mathbb{C})$, de régularité $\mc{C}^2$.
Supposons que $N$ admette une stratification $N=\cup N_i$ vérifiant la condition (C).
Alors $p^*N$ est une sous-variété de régularité $\mc{C}^1$ dans $\mathbb{P}(T \mathbb{P}^2(\mathbb{C}))$ et admet une stratification dont les strates de dimension 3 sont semi-legendriennes.
De plus, $p(p^*N)=N$.
\end{thm}

\begin{proof}
Au vu des résultats des sections précédentes, il suffit de voir que les relevés $p^*N_i$ des différentes strates se recollent bien, c'est à dire que les $p^*N_i$ et leurs intersections forment une stratification d'une sous-variété.
Comme $p^*N_i$ est le fibré entier $p^{-1}(N_i)$ si $\mr{dim}(N_i)\leq 1$, et au vu de la condition (C), il n'y a qu'à vérifier les paires de strates $N_i<N_j$ où $\mr{dim}(N_i)=2$.
Mais si un hyperplan $H$ contient un 2-plan $P$, alors $T^{1,0}H$ est une direction complexe critique par rapport à $P$.
Ainsi par la condition (C) l'intersection $\overline{p^*N_j}\cap p^*N_i$ est une section du fibré en cercles $p^*N_i$, ce qui conclut la preuve.
\end{proof}

\begin{thm}
\label{thm_stratification_2}
Soit $M$ une sous-variété réelle de dimension trois, de régularité $\mc{C}^2$ dans $\mathbb{P}(T \mathbb{P}^2(\mathbb{C}))$.
On suppose que $N:=p(M)$ est une sous-variété de $\mathbb{P}^2(\mathbb{C})$, et que $M$ et $N$ admettent des stratifications $M=\cup M_i$ et $N=\cup N_i$ de sorte que les strates de dimension 3 de $M$ soient semi-legendriennes et que $(N_i)$ vérifie la condition (C).
On suppose de plus que $p$ envoie toute strate de $M$ dans une strate de $N$, et est une submersion en restriction à chaque strate.
Alors $M\subset p^*(p(M))$.
\end{thm}

\begin{proof}
D'après les hypothèses, pour chaque strate $M_i$ de dimension 3 et chaque $w\in M_i$, le germe $(M_i,w)$ muni de la projection $p$ est dans l'une des situations décrites dans les sous-sections précédentes, donnant le résultat sur ces strates.
Pour les strates $M_i$ telles que $\mr{dim}(p(M_i))\leq 1$, on a $p^*(p(M_i))=p^{-1}(p(M_i))\supset M_i$.
Si $M_i$ est une strate de dimension 2 transverse à $p$, il existe une strate $M_j$ de dimension 3 avec $M_i<M_j$ du fait que $M$ est de dimension 3.
Soit $w\in M_i$ et $w_n\in M_j$ une suite de points tels que $w_n\to w$.
Notons $w_n=(z_n,\lambda_n)$ et $w=(z,\lambda)$.
Par le caractère semi-legendrien de $M_j$, le nombre $\lambda_n$ est la pente en $z_n$ de la strate de $N$ correspondante.
Par la condition (C), on déduit que $\lambda$, qui est égal à la limite des $\lambda_n$, est une direction critique pour $p(M_i)$ en $z$.
Le résultat s'ensuit.
\end{proof}

\section{Dualité dans $\mathbb{P}^2(\mathbb{C})$}
\label{sec_dualite}

\subsection{Généralités}

Considérons l'espace 
\[
X = \mathbb{P}(T \mathbb{P}^2)= \{ (z,d)\in \mathbb{P}^2\times \check{\mathbb{P}}^2, z\in d\}
\]
muni des projections $p: X \rightarrow \mathbb{P}^2$ vers l'espace ambient et $\pi: X \rightarrow \check{\mathbb{P}}^2$ vers l'espace dual, et de la structure de contact commune aux deux projections.
À une sous-variété $N\subset \mathbb{P}^2$ de régularité $\mc{C}^1$ on peut associer la variété $p^*N$ définie en section \ref{sec_semi_legendrienne}.
On définit alors le \emph{dual} $\check{N}$ de $N$ comme étant 
\[
\check{N} := \pi(p^*N).
\]

\begin{lem}
\label{lem_bidualite_generique}
Soit $N$ un germe de sous-variété lisse générique de régularité $\mc{C}^2$.
Alors son dual $\check{N}$ est lisse de régularité $\mc{C}^1$ et vérifie $N=\check{\check{N}}$.
\end{lem}

\begin{proof}
Pour un germe de sous-variété générique $N$, la sous-variété semi-legendrienne $p^*N$ est lisse, $\mc{C}^1$ et transverse à $\pi$.
D'après les résultats de la section \ref{sec_semi_legendrienne}, on a $\pi^*(\pi(p^*N)) = p^*N$, c'est à dire $\pi^*\check{N}=p^*N$.
De plus $p(p^*N)=N$, d'où le résultat.
\end{proof}

Voyons ce que cela dit dans le cas où $N=S$ est une surface.

Si $S$ est une surface, $\check{S}$ est génériquement de dimension réelle 3, mais peut contenir des portions de dimension 2 si $S$ contient des parties holomorphes ; un point de $S$ dont la tangente est complexe correspond à un point singulier de $\check{S}$ ; et chaque fois qu'il existe un germe de courbe réel $C\subset S\cap D$ pour une droite complexe $D$, cette droite $D$ correspond à un point singulier de $\check{S}$.

\begin{rmq}
\label{rmq_exceptionnel}
On voit que les points $(z,\lambda)\in p^*S$ qui sont critiques pour la projection $\pi$ sont exactement les points exceptionnels au sens de la définition \ref{df_exceptionnel}.
\end{rmq}

En général le dual d'une hypersurface $H\subset \check{\mathbb{P}}^2$ sera de dimension 3.
On remarque cependant que si $d\in H$ et $z= T^{1,0}_dH$, alors l'intersection $z\cap H$ est génériquement de dimension 1 et singulière en $d$ ; mais que dans le cas où $H=\check{S}$ est le dual d'une surface, la situation est particulière.
En effet, $p^*S$ est fibré par des cercles, donc génériquement $\pi(p^*S)$ admet une fibration par courbes $C_t$.
Cependant, par définition la variété $\pi^*H$ est une section du fibré en $\mathbb{P}^1$ au-dessus de $H$, et celle-ci est constante égale à $z$ au-dessus de $C_t$.
On en déduit que les courbes $C_t$ sont exactement les intersections entre les droites complexes tangentes et $H$.

\begin{rmq}
Cette propriété est à comparer à la propriété que l'enveloppe de la famille des droites tangentes à une courbe réelle $C\subset \mathbb{P}^2(\mathbb{R})$ est $C$ elle-même.
En effet, dans le cas ci-dessus, on peut aussi vérifier que si $S$ est un germe de surface lisse de tangente non complexe et si $\lambda$ est une direction critique non exceptionnelle, alors pour toute courbe $c(t)$ sur $S$ passant par l'origine dans la direction $\lambda$, et toute famille de droites non transverses $d(t)$ en $c(t)$ telle que $d(0)$ est dans la direction $\lambda$, alors l'intersection $d(0)\cap d(t)$ tend vers l'origine quand $t\to 0$.
Ceci peut être vérifié en calculant explicitement $\Omega_S \wedge \Omega_{\lambda(t)}$ comme dans le lemme \ref{lem_exceptionnel}.

Comme l'hypersurface duale $\check{S}$ contient aussi les cercles $C_t$ définis ci-dessus, cette remarque permet de reprouver que la droite $z$ dans $\check{\mathbb{P}}^2$ est la tangente complexe à $\check{S}$.
\end{rmq}

Appelons \emph{exceptionnelle} une hypersurface $H\subset \mathbb{P}^2$ telle que pour tout $z\in H$, si $d=T^{1,0}_zH$, alors pour tout $z'$ dans l'intersection $d\cap H$ proche de $z$, la droite $d$ est toujours tangente à $H$ en $z'$.
On a donc vu que si $H$ est le dual d'une surface, elle est exceptionnelle ; on voit aussi que si $H$ est exceptionnelle, la dimension de $p^*H$ chute lors de la projection par $\pi$, et que le dual $\check{H}$ est de dimension au plus 2 dans $\check{\mathbb{P}}^2$.

\begin{lem}
\label{lem_bidualite}
Supposons que $(S,z_0)\subset \mathbb{P}^2$ soit un germe de surface réelle lisse de régularité $\mc{C}^2$ tel que $T_{z_0}S$ ne soit pas complexe et tel qu'il existe une droite complexe non exceptionnelle $d_0$ dans le cercle critique.
Alors le germe dual $(\check{S},d_0)$ est lisse de régularité $\mc{C}^1$, et son dual $\check{\check{S}}$ vérifie $S=\check{\check{S}}$.
\end{lem}

\begin{proof}
C'est une version du lemme \ref{lem_bidualite_generique} rendant plus précise la notion de généricité en utilisant la remarque \ref{rmq_exceptionnel}.
\end{proof}

Cette propriété de bidualité permet de conclure immédiatement qu'il est impossible que deux germes de surfaces réelles aient le même dual.
Par suite, il est impossible qu'une surface admette une famille à trois paramètres de droites bicritiques (propriété dont la preuve était laissée à la sagacité du lecteur dans \cite{pohl_ordnungsgeometrie}).

\subsection{Application : algébricité des surfaces réelles}

On voit que si la surface réelle $S$ est compacte sans bord et suffisamment générique, son dual $\check{S}$ sera une hypersurface réelle compacte sans bord de $\check{\mathbb{P}}^2$, de sorte que $\check{\mathbb{P}}^2\setminus \check{S}$ se décompose en une union disjointe
\[
\check{\mathbb{P}}^2\setminus \check{S} = \bigcup V_i,
\]
où le nombre d'intersection entre une droite de $V_i$ et $S$ est un nombre constant $n_i$.
Si deux ouverts $V_i,V_j$ admettent une frontière commune $H\subset \check{S}$, alors le nombre d'intersection saute de deux lors du passage de la frontière : $n_j = n_i \pm 2$.

L'hypothèse de généricité de la surface $S$ est à comprendre au sens de la section \ref{sec_stratifications} : on veut que $p^*S$ soit suffisamment transverse à la projection $\pi$.
Cependant, cette hypothèse n'est ici utile que pour que le dual soit une belle hypersurface globale ; l'argument du passage de frontière n'a pas besoin d'une telle généricité pour fonctionner :

\begin{thm}
\label{thm_principal}
Soit $S$ une surface réelle connexe compacte irréductible de $\mathbb{P}^2(\mathbb{C})$ de régularité $\mc{C}^2$.
Alors ou bien
\begin{enumerate}
\item\label{item_1} $S$ est une courbe algébrique complexe, ou
\item\label{item_2} 
$S$ est le compactifié d'un plan affine réel, ou
\item\label{item_3} il existe deux ouverts non vides $V_1,V_2\subset \check{\mathbb{P}}^2$ tels que toute droite de $V_i$ intersecte $S$ transversalement en $n_i$ points, et $n_1\neq n_2$.
\end{enumerate}
\end{thm}

\begin{proof}
Les points de $S$ peuvent être de quatre types : les points singuliers $\mr{Sing}(S)$, qui forment un sous-ensemble fermé de dimension au plus 1 ; les points à tangence complexe, $F_{\mathbb{C}}$ qui forment un sous-ensemble fermé ; les points $z$ tels que tout couple $(z,\lambda)$ critique soit exceptionnel, que l'on regroupe en un ensemble $F_{ex}$ fermé ; et les points qui se comportent régulièrement pour la dualité.
D'après les hypothèses et le corollaire \ref{cor_R2}, si $S$ est l'union de $F_{\mathbb{C}}$, $F_{ex}$ et $\mr{Sing}(S)$, alors on est dans l'un des cas \eqref{item_1} ou \eqref{item_2}.
En effet, la surface $\mathbb{R}^2$ a un angle holomorphe $\theta=\pi/2$ et une courbe complexe a un angle $\theta=0$ ; on ne peut pas recoller les deux de manière $\mc{C}^1$.
On peut donc supposer qu'il existe un point $z_0\in S$ lisse dont le plan tangent n'est pas complexe et tel qu'il existe en $z_0$ une direction critique non exceptionnelle.

Considérons le fermé $F = F_{\mathbb{C}}\cup F_{ex}\cup\mr{Sing}(S)$.
Son dual $\check{F}$ est l'image par $\pi$ d'un compact, donc est aussi fermé ; il ne contient pas $\check{S}$ car les duaux de $F_{\mathbb{C}}$ et $F_{ex}$ sont de dimension 2 et le dual de $\mr{Sing}(S)$ est feuilleté par des droites complexes.
Par suite il existe $\varepsilon > 0$ tel que l'ensemble 
\[
\check{F}^{(\varepsilon)} := \{l\in \check{\mathbb{P}}^2, d(l, \check{F}) < \varepsilon \}
\]
ne contienne pas $\check{S}$, où la distance $d$ est la métrique de Fubini-Study.
Alors $\check{S}^{(\varepsilon)} := \check{S}\setminus \check{F}^{(\varepsilon)}$ est compact ; d'après le lemme \ref{lem_bidualite}, il existe un nombre fini de petits ouverts $S_i\subset S$ tels que les $\check{S}_i$ soient lisses et recouvrent $\check{S}^{(\varepsilon)}$.

Pour tout $i$, il est impossible que l'intersection $\check{S}_i\cap (\cup_{j\neq i} \check{S}_j)$ soit dense dans $\check{S}_i$, car sinon, par finitude il existerait un $j$ tel que $\check{S}_i\cap \check{S}_j$ soit dense dans $\check{S}_i$ (quitte à restreindre les $S_i$).
Ceci impliquerait $\check{S}_i\subset \check{S}_j$, puis $S_i\subset S_j$ d'après le lemme \ref{lem_bidualite}.

Ainsi il existe une petite boule ouverte $V\subset \check{\mathbb{P}}^2\setminus \check{F}^{(\varepsilon)}$ tel que $V\cap \check{S}_j\neq \emptyset$ si et seulement si $j=i$.
Par suite $V\setminus \check{S} = V_+\cup V_-$ se décompose en deux, et d'après les discussions de la section \ref{sec_rappels}, les nombres d'intersection $n_+,n_-$ entre une droite de $V_+$ ou $V_-$ et $S$ sont constants sur $V_+$ et $V_-$, et $n_+=n_-+2$.
\end{proof}

\begin{rmq}
Considérons la surface affine réelle $\mathbb{P}^2(\mathbb{R})\subset \mathbb{P}^2(\mathbb{C})$.
On voit que $\mathbb{P}^2(\mathbb{R})$ est son propre dual donc que le complémentaire du dual est connexe ; le nombre d'intersection générique entre $\mathbb{P}^2(\mathbb{R})$ et une droite complexe doit ainsi être constant.
Évidemment, ce nombre est 1 : c'est le nombre de points réels d'une droite complexe générale.
\end{rmq}

\begin{cor}
\label{cor_principal}
Soit $V$ une sous-variété réelle connexe compacte irréductible de dimension réelle $2n$ et de régularité $\mc{C}^2$ de $\mathbb{P}^{n+k}(\mathbb{C})$, et $G$ la grassmannienne des k-plans linéaire complexes.
Alors ou bien
\begin{enumerate}
\item $V$ est une sous-variété algébrique complexe, ou
\item $V$ est le compactifié d'une sous-variété affine réelle de $\mathbb{C}^{n+k}$, ou
\item il existe deux ouverts non vides $U_1, U_2\subset G$ tels que tout plan de $U_i$ intersecte $V$ transversalement en $n_i$ points, et $n_1\neq n_2$.
\end{enumerate}
\end{cor}

\begin{proof}
On se ramène à $n=k=1$ comme dans \cite{thom_varietes_ordre_fini}, en considérant des sections planes puis des projections.
On vérifie aisément que les seules variétés qui donnent des plans réels ou des courbes complexes pour n'importe quelle section plane et n'importe quelle projection sont les sous-variétés affines réelles et les sous-variétés algébriques complexes.
\end{proof}

\bibliography{dualite}{}
\bibliographystyle{acm}

\textsc{IMPA, Estrada Dona Castorina, 110, Horto, Rio de Janeiro, Brasil}

\textit{Email :} olivier.thom@impa.br

\end{document}